\documentclass[10pt, reqno]{amsart}
\usepackage{graphicx}
\usepackage{fullpage}
\usepackage{fancyhdr} 
\usepackage{setspace}
\usepackage{tikz}

\usepackage{amsthm}
\usepackage{amsmath}
\usepackage{amssymb}
\usepackage{enumerate}
\usepackage{graphicx}
\usepackage{cancel}
\usepackage{hyperref}

\newtheorem{thm}{Theorem}[section]

\newtheorem{lem}[thm]{Lemma}
\newtheorem{cor}[thm]{Corollary}
\theoremstyle{definition}

\newtheorem{defn}[thm]{Definition}
\newtheorem{conj}[thm]{Conjecture}
\newtheorem{qstn}[thm]{Question}

\theoremstyle{remark}
\newtheorem{lem*}{Lemma}

\input xy
\xyoption{all}
\title{Minimum Degrees of Minimal Ramsey Graphs for Almost-Cliques}
\author{Andrey Grinshpun}
\thanks{The first author's research is partially supported by a National Physical Sciences Consortium Fellowship}
\author{Raj Raina}
\author{Rik Sengupta}
\date{}
\begin{document}

\begin{abstract}
For graphs $F$ and $H$, we say $F$ is \textit{Ramsey for $H$} if every $2$-coloring of the edges of $F$ contains a monochromatic copy of $H$. The graph $F$ is \textit{Ramsey $H$-minimal} if $F$ is Ramsey for $H$ and there is no proper subgraph $F'$ of $F$ so that $F'$ is Ramsey for $H$. Burr, Erd\H os, and Lov\' asz defined  $s(H)$ to be the minimum degree of $F$ over all Ramsey $H$-minimal graphs $F$. Define $H_{t,d}$ to be a graph on $t+1$ vertices consisting of a complete graph on $t$ vertices and one additional vertex of degree $d$. We show that $s(H_{t,d})=d^2$ for all values $1<d\le t$; it was previously known that $s(H_{t,1})=t-1$, so it is surprising that $s(H_{t,2})=4$ is much smaller.

We also make some further progress on some sparser graphs. Fox and Lin observed that $s(H)\ge 2\delta(H)-1$ for all graphs $H$, where $\delta(H)$ is the minimum degree of $H$; Szab\'o, Zumstein, and Z\"urcher investigated which graphs have this property and conjectured that all bipartite graphs $H$ without isolated vertices satisfy $s(H)=2\delta(H)-1$. Fox, Grinshpun, Liebenau, Person, and Szab\'o further conjectured that all triangle-free graphs without isolated vertices satisfy this property. We show that $d$-regular $3$-connected triangle-free graphs $H$, with one extra technical constraint, satisfy $s(H) = 2\delta(H)-1$; the extra constraint is that $H$ has a vertex $v$ so that if one removes $v$ and its neighborhood from $H$, the remainder is connected.
\end{abstract}

\maketitle 

\section{Introduction}
If $F$ and $H$ are finite graphs, we write $F\to H$ and say $F$ is \textit{Ramsey for $H$} to mean that every $2$-coloring of the edges of $F$ with the colors red and blue contains a monochromatic copy of $H$. For any fixed graph $H$, the collection of graphs that are Ramsey for it is upwards closed; that is, if $F'$ is a subgraph of $F$ and $F'$ is Ramsey for $H$, then $F$ is also Ramsey for $H$. Therefore, in order to understand the collection of graphs that are Ramsey for $H$, it is sufficient to understand the graphs that are minimal with this property; we call these graphs \textit{Ramsey $H$-minimal}, or \textit{$H$-minimal} for short, and denote the collection of these Ramsey $H$-minimal graphs by $\mathcal{M}(H)$. One of the foundational results in Ramsey theory, Ramsey's theorem, states that for all graphs $H$, the set $\mathcal{M}(H)$ is nonempty \cite{fpr}.

The fundamental goal of graph Ramsey theory is to understand the properties of the graphs in the family $\mathcal{M}(H)$, given the graph $H$. Several questions about the extremal properties of graphs in $\mathcal{M}(H)$ have been asked throughout the years. One of the most famous such questions is the Ramsey number of $H$, denoted by $r(H)$, which asks for the minimum number of vertices of any graph in $\mathcal{M}(H)$. This number is only known for very few classes of graphs $H$. Of particular interest is $r(K_t)$ ($K_t$ is the complete graph on $t$ vertices), which is known to be at least $2^{t/2}$ \cite{e} and at most $2^{2t}$ \cite{es}. Despite these bounds being over 60 years old, the constants in the exponents have not been improved, making this one of the oldest and most difficult open problems in combinatorics. The study of $\mathcal{M}(H)$ has also extended in various other directions. In this paper, we are interested in the following value, first studied by Burr, Erd\H os, and Lov\'asz \cite{bel}:
\begin{equation*}
s(H):=\displaystyle \min_{F\in \mathcal{M}(H)} \delta(F)
\end{equation*}
where $\delta(F)$ is the minimum degree of $F$. Because of the $H$-minimality condition imposed on $F$, one cannot arbitrarily add vertices of small degree to $F$.

Define $H_{t,d}$ to be the graph on $t+1$ vertices which consists of a clique on $t$ vertices and an additional vertex of degree $d$. In \cite{bel} it is shown that $s(H_{t,t})=t^2$, in \cite{szz} it is shown that $s(H_{t,0})=(t-1)^2$, and in \cite{fglps} it is shown that $s(H_{t,1})=t-1$. We find $s(H_{t,d})$ for all $1 < d < t$, showing that
$$
s(H_{t,d})=
\begin{cases}
d^2 & \text{if }1<d < t\\
t-1 & \text{if }d = 1\\
(t-1)^2 & \text{if }d = 0.
\end{cases}
$$
The discrepancy between the values of $s(H_{t,0})$ and $s(H_{t,1})$ was already known, but the discrepancy between $s(H_{t,1})$ and $s(H_{t,2})$ is perhaps more surprising, as both graphs are connected. It is also interesting to note that, if we take $d$ large enough compared to $t$, then the resulting graphs are the first time $s(H)$ has been determined for very well-connected graphs which are not vertex-transitive; much work has been focused around computing $s(H)$ where $H$ is either a sparse graph or is vertex-transitive, which are somewhat easier cases to handle.

Graphs $H$ that satisfy $s(H)=2\delta(H)-1$ are called Ramsey simple. In \cite{szz} it is shown that many bipartite graphs satisfy $s(H)=2\delta(H)-1$, including forests, even cycles, and connected, balanced bipartite graphs (a bipartite graph is balanced if both parts have the same size). It was further conjectured that all bipartite graphs without isolated vertices are Ramsey simple. In \cite{fglps2}, the authors show that all $3$-connected bipartite graphs are Ramsey simple. They also show that in any $3$-connected graph $H$, if there is a minimum-degree vertex $v$ so that its neighborhood is contained in an independent set of size $2\delta(H)-1$, then $s(H) = 2\delta(H) - 1$. They further conjectured that all triangle-free graphs without isolated vertices are Ramsey simple. In this paper, we prove that any $d$-regular $3$-connected triangle-free graph, with one additional technical constraint, is Ramsey simple. These constraints are not severely restrictive, since a random $d$-regular triangle-free graph (for a fixed constant $d\ge 3$) satisfies all of these constraints with high probability.

This paper is arranged as follows. In Section \ref{DefAndBound} we introduce the notation necessary for the paper and some known simple bounds on $s(H)$. In Section \ref{CompleteWithAddedVertex} we compute the exact value of $s(H)$ for the graphs $H_{t, d}$ for all $0 \leq d \leq t$, expanding on the results of \cite{bel}, \cite{fglps}, and \cite{szz}. In Section \ref{TriangleFree} we find a new class of Ramsey simple graphs. Finally in Section \ref{openproblems} we wrap up with some open questions and directions of further research. This work builds on the findings and techniques of \cite{fglps2}, \cite{fglps}, \cite{fl}, and \cite{szz}.

\section{Preliminaries and background}\label{DefAndBound}
\subsection{Standard Definitions}
Given a graph $H$, the \textit{neighborhood} of a vertex $v\in V(H)$, denoted by $N(v)$, is the set of all vertices in $H$ that are adjacent to $v$ and the \textit{degree} of $v$, denoted by $\operatorname{deg}{(v)}$, is the size of its neighborhood. A graph is \textit{regular} if all vertices have the same degree and it is $d$-regular if all vertices have degree $d$. The \textit{independence number} of a graph $\alpha (H)$ is defined as the size of the largest set of vertices in $H$ that induces an independent set in $H$ (a set that contains no edges), and the \textit{clique number} of a graph $\omega(H)$ is the size of the largest clique in $H$.

Define $G \boxtimes H$ to be the graph obtained by taking disjoint copies of $G$ and $H$ and adding a complete bipartite graph between them. When we write $G_1 \boxtimes G_2 \boxtimes G_3$, we mean there is a complete bipartite graph between \textit{every} pair of the graphs $G_1$, $G_2$, and $G_3$, not just between the pairs $(G_1, G_2)$ and $(G_2, G_3)$.

\subsection{Simple bounds}
The following are simple bounds on $s(H)$.
\begin{thm}[\cite{fl} and \cite{bel}]\label{TrivialBound}
For all graphs $H$, we have 
\begin{center}$2\delta (H)-1 \le s(H) \le r(H)-1$.\end{center}
\end{thm}
\begin{proof}
For the lower bound, suppose $s(H)< 2\delta(H)-1$. Consider a graph $F\in M(H)$ with a vertex $v$ of degree $s(H) < 2 \delta(H)-1$. By minimality, there must be some coloring of $F - v$ without a monochromatic copy of $H$. We extend this to a coloring of $F$. To do this, partition $N(v)$ into two sets, $R(v)$ and $B(v)$, so that $|R(v)|\le \delta(H)-1$ and $|B(v)|\le  \delta(H)-1$. For any $x \in R(v)$ color the edge $\{v,x\}$ red and for any $y \in B(v)$ color the edge $\{v,y\}$ blue. In such a coloring, $v$ can never be a part of a monochromatic copy of $H$, since its degree in that copy would be less than $\delta(H)$, a contradiction.

For the upper bound, simply note that by definition there is a graph on $r(H)$ vertices that is Ramsey $H$-minimal. Any vertex in this graph has degree at most $r(H)-1$, yielding the desired bound.
\end{proof}
The lower bound has been shown to be exact for all $3$-connected bipartite graphs \cite{fglps2} and some other classes of bipartite graphs \cite{szz}. However, for many graphs $H$, the upper bound is much larger than $s(H)$; $r(H)$ may be exponentially large in the number of vertices of $H$ \cite{e}, while all known values of $s(H)$ are bounded by a polynomial in the number of vertices of $H$.

\subsection{BEL gadgets}

The following theorem is used for all of the results in the paper, so we state it here. It roughly states that, for any $3$-connected graph $H$, we can find a graph $F$ that, if its edges are $2$-colored in such a way that there is no monochromatic copy of $H$, we can force whatever color pattern we want in a certain region of $F$.

\begin{thm}[\cite{bnr}] \label{BELExtension}
Given any $3$-connected graph $H$, any graph $G$, and any $2$-coloring $\psi$ of $G$ without a monochromatic copy of $H$, there is a graph $F$ with the following properties:
\begin{enumerate}
\item $F\not \to H$,
\item $F$ contains $G$ as an induced subgraph, and
\item for any $2$-coloring of $F$ without a monochromatic copy of $H$, the coloring $G$ agrees with $\psi$, up to permutation of the two colors.
\end{enumerate}
\end{thm} 
We call a graph $F$ with coloring $\psi$ and induced subgraph $G$ constructed in this manner a \textit{BEL gadget}, and if $H$ satisfies the conclusions of the above theorem for all $G$ and $\psi$, we say $H$ has BEL gadgets. In particular, it is shown in \cite{bnr}  that all $3$-connected graphs have BEL gadgets. Note that the acronym BEL stands for Burr, Erd\H{o}s, and Lov\'asz, who first proved the existence of such gadgets for $H=K_t$ \cite{bel}.

\section{The complete graph with an added vertex}\label{CompleteWithAddedVertex}
Recall that $H_{t,d}$ is the graph on $t+1$ vertices that contains a $K_t$ and in which the remaining vertex (not in the $K_t$) has degree $d$, with its neighbors being any $d$ vertices of the $K_t$.

Note $H_{d,d}$ is isomorphic to $K_{d+1}$, for which $s(K_{d+1})$ is known to be $d^2$ \cite{bel}. For $d=1$, it was recently shown that $s(H_{t,1})=t-1$ \cite{fglps}. For $d=0$, it was found $s(H_{t,0})=s(K_t)=(t-1)^2$ \cite{szz}. A natural question that arises is how $s(H_{t,d})$ behaves when $d$ is between $1$ and $t$. We now state the main result of this section.
\begin{thm}\label{Htd}
For all $1< d < t$ we have
\[s(H_{t,d})=d^2.\]
\end{thm}
The proof of this theorem is presented in two parts. In the first part, we prove that $s(H_{t,d})\ge d^2$ for all values of $d$. The second part expands on the ideas in \cite{bel} and \cite{fglps} and deals with the upper bound on $s(H_{t,d})$ for $d\ge 2$: we construct a graph $G$ with a vertex $v$ of degree $d^2$ that is Ramsey for $H_{t,d}$ such that $G-v\not \to H_{t,d}$. It follows from this that $s(H_{t,d})\le d^2$, and so we obtain $s(H_{t,d})=d^2$ for all $1< d < t$. We now begin with the first part of our proof, which closely follows the ideas of \cite{bel}.
\begin{lem}
Let $H$ be a graph such that, for all $v \in V(H)$, the neighborhood of $v$ contains a copy of $K_d$. Then
$s(H)\ge d^2$.
\end{lem}
\begin{proof}
Suppose there exists $F \in \mathcal{M}(H)$ and some $v \in V(F)$ with $\operatorname{deg}{(v)}<d^2$. Since $F$ is minimal, we can $2$-color the edges of $F-v$ so that there is no monochromatic copy of $H$. Consider any such $2$-coloring of $F-v$. In this coloring, let $S$ denote the neighborhood of $v$ and let $T_1, \dots ,T_k$ be a maximal set of vertex-disjoint red copies of $K_{d}$ in $S$. Since $\operatorname{deg}{(v)}<d^2$, we must have $|S|<d^2$, and so $k\le d-1$. Now we color all the edges connecting $v$ to $T_1,\dots ,T_k$ blue, and all other edges incident to $v$ red. We claim that no monochromatic copy of $H$ arises in such a coloring. Note that such a copy would need to use $v$. We will now show that there is no red $d$-clique in the red neighborhood of $v$ and that there is no blue $d$-clique in the blue neighborhood of $v$, thus showing that $v$ cannot be contained in any monochromatic copy of $H$.

Any red $d$-clique in $S$ must intersect one of $T_1,\ldots,T_k$ and therefore would have a blue edge from $v$. On the other hand, suppose there exists a blue $d$-clique in the blue neighborhood of $v$, which is precisely $T_1 \cup \cdots \cup T_k$. Since $k\le d-1$, by the pigeonhole principle, at least two vertices of this blue $d$-clique must be contained in the same $T_i$. These two vertices, however, are connected by a red edge, a contradiction. It follows that such an $F \in \mathcal{M}(H)$ cannot exist, and hence $s(H) \ge d^2$. 
\end{proof}
Since the neighborhood of each vertex in $H_{t,d}$ contains a copy of $K_d$, we have the following corollary.
\begin{cor}\label{LowerBoundHtd}
For all values of $d$ we have $s(H_{t,d})\ge d^2$.
\end{cor}
This completes the first part of our proof, establishing a lower bound on the value of $s(H_{t,d})$.

For the upper bound, we wish to construct an $H$-minimal graph with vertex of degree exactly $d^2$ for $d \ge 2$. To that end, we wish to show that $H_{t,d}$ has BEL gadgets. Theorem \ref{BELExtension} implies this in the case $d \geq 3$, but not when $d = 2$; the majority of the work in this section is proving that $H_{t,2}$ has BEL gadgets.

\begin{thm}\label{BELHtd}
For all $2 \leq d \leq t$, the graph $H_{t,d}$ has BEL gadgets.
\end{thm}

We postpone the proof of this theorem to the end of the section; let us first see why it implies the desired upper bound on $s(H_{t,d})$.

\begin{lem} \label{Htd3}
For all $2 \leq d \leq t$ there exists a graph $F'$ with vertex $v$ of degree $d^2$ so that $F'\to H_{t,d}$ but $F'-v \not \to H_{t,d}$.
\end{lem}
\begin{proof}
If $d=t$ then $s(H_{t,d})=d^2$ by \cite{bel}, which immediately implies the lemma; we will henceforth assume $d < t$.

The graph $H_{t,d}$ has BEL gadgets by Theorem \ref{BELHtd}. This means that, for any graph $G$ and $2$-coloring $\psi$ of $G$ without a monochromatic copy of $H$, there exists a graph $F\not \to H_{t,d}$ with an induced copy of $G$ such that every $2$-coloring of $F$ without a monochromatic copy of $H_{t,d}$ agrees with $\psi$ on the copy of $G$, up to permutation of colors. We describe our graph $G$ together with its coloring $\psi$ for our BEL gadget as follows: 
\begin{enumerate}
\item $G$ contains $d$ disjoint red copies $T_1,\dots ,T_{d}$ of $K_{t}$,
\item For each distinct pair $i$ and $j$, there is a complete blue bipartite graph between $T_i$ and $T_j$, and
\item For each way there is to choose a $d$-tuple $T = (t_1, \ldots, t_d) \in T_1 \times \cdots \times T_d$ by taking one vertex from each $T_i$, we add a set of $t-d$ vertices $S_T=\{v_1^T, \ldots ,v_{t-d}^T\}$; we add blue edges between all pairs of vertices in $S_T$ so that $S_T$ becomes a blue clique, and add more blue edges so that there is a complete blue bipartite graph between $S_T$ and $T$. For distinct $d$-tuples $T$ and $T'$, $S_T$ and $S_{T'}$ are disjoint.
\end{enumerate}
An example of this $G$ with coloring $\psi$ is shown in Figure \ref{fig1}. We first claim that this coloring $\psi$ contains no monochromatic copy of $H_{t,d}$. The connected components in red are all copies of $K_t$, so there is no red copy of $H_{t,d}$. We also claim there is no blue copy of $H_{t,d}$. If we omit the vertices that are contained in the various $S_T$, the blue graph is $d$-partite and so contains no $K_t$, as $d < t$. Therefore, any blue copy of $H_{t,d}$ must use some vertex $w$ in some $S_T$ as part of a blue $K_t$. Note that the blue degree of $w$ is $t-1$, and therefore this blue $K_t$ must consist precisely of $w$ and its neighborhood. However, any vertex that is not $w$ or contained in the blue neighborhood of $w$ has degree at most $d-1$ to the neighborhood of $w$ by construction, and so cannot be the vertex of degree $d$ in $H_{t, d}$. Therefore, there is no blue copy of $H_{t,d}$.

Consider a graph $F\not \to H_{t,d}$ with an induced copy of $G$ such that any $2$-coloring of $F$ without a monochromatic copy of $H_{t,d}$ restricts to the coloring $\psi$ on the induced copy of $G$, up to permutation of the colors; this exists by Theorem \ref{BELHtd}. We now modify $F$ to $F'$ by adding a vertex $v$, and adding $d$ edges from $v$ to each $T_i$ in the induced copy of $G$. The vertex $v$ clearly has degree $d^2$. We claim that this modified graph $F'$ is Ramsey for $H_{t,d}$. Consider any $2$-coloring of $F'$. In this $2$-coloring, if there is a monochromatic copy of $H_{t, d}$ in the subgraph $F = F' - v$, then we are done. Otherwise suppose the $2$-coloring does not yield a monochromatic copy of $H_{t,d}$ in $F$. Then the induced graph $G$ must have coloring $\psi$, up to permutation of colors. Let us assume without loss of generality that each $T_i$ forms a red clique and the remaining edges are blue.

If $v$ had red degree $d$ to some $T_i$, then $v$ together with $T_i$ would be a red copy of $H_{t,d}$. Thus, at least one edge from $v$ to each copy of $T_i$ must be colored blue. Choose one vertex $t_i$ from each $T_i$ so that $v$ has a blue edge to $t_i$ and take $T = (t_1,\ldots,t_d)$. Then these vertices $t_i$ together with $S_T$ forms a blue $K_t$, and adding $v$ creates a blue $H_{t,d}$.
\end{proof}

This immediately gives the desired upper bound on $s(H_{t,d})$.

\begin{cor}
For every $2 \leq d \leq t$, we have $s(H_{t,d}) \leq d^2$.
\end{cor}
\begin{proof}
By the previous lemma, there is a graph $F'$ with a vertex $v$ of degree $d^2$ which is Ramsey for $H_{t,d}$ so that $F' - v$ is not Ramsey for $H_{t,d}$. Take $F''$ to be a subgraph of $F'$ which is minimal subject to the constraint that $F''$ is Ramsey for $F$. $F''$ must contain $v$, and so $s(H_{t,d}) \leq\delta(F'') \leq d^2$, as desired.
\end{proof}

\begin{figure}[ht]
\begin{center}
\begin{tikzpicture}

\draw[blue,dashed,very thick] (-4,0.5) -- (4,0.5);
\draw[blue,dashed,very thick] (-4,0.8) -- (4,0.8);
\draw[blue,dashed,very thick] (-4,1.1) -- (4,1.1);

\draw[blue,dashed,very thick] (-4,0.5) to[out=-60,in=-120] (4,0.5);
\draw[blue,dashed,very thick] (-4,0.8) to[out=-60,in=-120] (4,0.8);
\draw[blue,dashed,very thick] (-4,1.1) to[out=-60,in=-120] (4,1.1);

\tikzstyle{every node}=[draw,circle,fill=white,minimum size=75pt,
                            inner sep=0pt]
														
		\draw (-4,0.85) node (A) {};
    \draw (0,0.85) node (B) {};
    \draw (4,0.85) node (C) {};

\tikzstyle{every node}=[draw,circle,fill=black,minimum size=4pt,
                            inner sep=0pt]
    \draw (-0.5,0) node (B4) {};
    \draw (0.5,0) node (B3) {};
    \draw (-1,1.1) node (B5) [label=below:$v_y$]{};
    \draw (1,1.1) node (B2) {};
    \draw (0,1.75) node (B1) {};
    
		\draw (-4.5,0) node (A4) {};
    \draw (-3.5,0) node (A3) {};
    \draw (-5,1.1) node (A5) {};
    \draw (-3,1.1) node (A2) [label=above:$v_x$]{};
    \draw (-4,1.75) node (A1) {};
    
		\draw (3.5,0) node (C4) {};
    \draw (4.5,0) node (C3) {};
    \draw (3,1.1) node (C5) [label=below:$v_z$]{};
    \draw (5,1.1) node (C2) {};
    \draw (4,1.75) node (C1) {};

		\draw (-1, 3.5) node (Y) [label=left:$S_T$]{};
		\draw (1, 3.5) node (Z) {};

		\draw[thick,red] (A1) -- (A2);
		\draw[thick,red] (A1) -- (A3);
		\draw[thick,red] (A1) -- (A4);
		\draw[thick,red] (A1) -- (A5);
		\draw[thick,red] (A2) -- (A3);
		\draw[thick,red] (A2) -- (A4);
		\draw[thick,red] (A2) -- (A5);
		\draw[thick,red] (A3) -- (A4);
		\draw[thick,red] (A3) -- (A5);
		\draw[thick,red] (A4) -- (A5);

		\draw[thick,red] (B1) -- (B2);
		\draw[thick,red] (B1) -- (B3);
		\draw[thick,red] (B1) -- (B4);
		\draw[thick,red] (B1) -- (B5);
		\draw[thick,red] (B2) -- (B3);
		\draw[thick,red] (B2) -- (B4);
		\draw[thick,red] (B2) -- (B5);
		\draw[thick,red] (B3) -- (B4);
		\draw[thick,red] (B3) -- (B5);
		\draw[thick,red] (B4) -- (B5);

		\draw[thick,red] (C1) -- (C2);
		\draw[thick,red] (C1) -- (C3);
		\draw[thick,red] (C1) -- (C4);
		\draw[thick,red] (C1) -- (C5);
		\draw[thick,red] (C2) -- (C3);
		\draw[thick,red] (C2) -- (C4);
		\draw[thick,red] (C2) -- (C5);
		\draw[thick,red] (C3) -- (C4);
		\draw[thick,red] (C3) -- (C5);
		\draw[thick,red] (C4) -- (C5);

		\draw[thick,blue] (A2) -- (Y);
		\draw[thick,blue] (A2) -- (Z);
		\draw[thick,blue] (B5) -- (Y);
		\draw[thick,blue] (B5) -- (Z);
		\draw[thick,blue] (C5) -- (Y);
		\draw[thick,blue] (C5) -- (Z);
		\draw[thick,blue] (Y) -- (Z);

\end{tikzpicture}
\caption{Example of $G$ with the coloring $\psi$ for $t=5$ and $d=3$. Here, only one set $S_T$ is shown, corresponding to the triple $T = (v_x, v_y, v_z)$. The dashed blue edges represent complete blue bipartite graphs. When we add the external vertex $v$, we will connect it to three vertices from each copy of $K_5$, making its degree $d^2 = 9$.}
\label{fig1}
\end{center}
\end{figure}
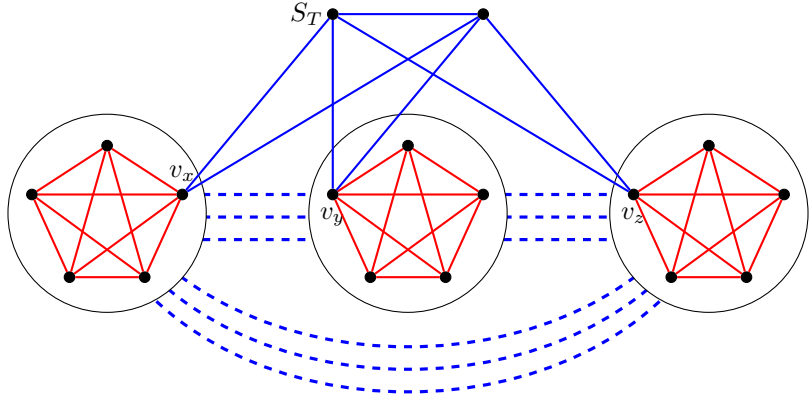

We now prove that $H_{t,2}$ has BEL gadgets. Note that, for $t=2$, the graph $H_{2,2}$ is isomorphic to $K_3$, for which it is known that BEL gadgets exist \cite{bel}. Henceforth, we will assume that $t\ge 3$. The ideas behind the proof of BEL gadgets for $H_{t,2}$ stems from a strategy in \cite{fglps}. We now introduce the main tool that we will need.

\begin{defn}
Write $F\overset{\epsilon}{\rightarrow} H$ to mean that, for every $S \subseteq V(F)$ such that $|S|\ge \epsilon |V(F)|$, the subgraph of $F$ induced by $S$ is Ramsey for $H$ (i.e. $F[S]\to H$).
\end{defn}
The following lemma, which is a strengthening of a theorem in \cite{nr}, is proven in \cite{fglps}.
\begin{lem} \label{EpsilonArrow}
For any graph $H$ and every $\epsilon>0$ and $t>2$, if $\omega(H) < t$ then there exists a graph $F$ that is $K_t$-free such that $F\overset{\epsilon}{\rightarrow} H$.
\end{lem}

We are now ready to construct a graph $G_0$ so that, for every coloring of $G_0$ without a monochromatic copy of $H_{t,2}$, a particular copy of some (arbitrary) graph $R_0$ is forced to be monochromatic. Furthermore, there is a coloring of $G_0$ where $R_0$ is red, all of the edges leaving $R_0$ are blue, there is no red $H_{t,2}$, and there is no blue $K_t$. The proof of this lemma closely follows the arguments in \cite{fglps}.
\begin{lem}\label{GadgetGraph}
Let $R_0$ be a graph that has no copy of $H_{t,2}$. Then there exists a graph $G_0$ with an induced copy of $R_0$ and the following properties:
\begin{enumerate}
\item There is a $2$-coloring of $G_0$ without a red copy of $H_{t,2}$ and without a blue copy of $K_t$ in which the edges of $R_0$ are red, and all of the edges incident to, but not contained in, $R_0$ are blue, and
\item Every $2$-coloring of $G_0$ without a monochromatic copy of $H_{t,2}$ results in $R_0$ being monochromatic.
\end{enumerate}
\end{lem}
\begin{proof}
Take $\epsilon = 2^{-n-t^2}$, where $n$ is the number of vertices in $R_0$. Let $F_1, F_2, \dots ,F_{t-2}$ be copies of the graph as defined in Lemma \ref{EpsilonArrow} when applied to $H=H_{t-1,1}$. We claim that the graph $G_0:= F_1 \boxtimes F_2 \boxtimes \dots \boxtimes F_{t-2}\boxtimes R_0$ satisfies both desired conditions (see Figure \ref{fig2}). 

To see the first property, color all the edges internal to any of $F_1, F_2, \dots, F_{t-2}, R_0$ red and the remaining edges blue. There can be no monochromatic red copy of $H_{t,2}$, since each $F_i$ is $K_t$-free and $R_0$ is $H_{t,2}$-free. Furthermore, there is no blue $K_t$, since the graph induced by the blue edges is $(t-1)$-chromatic.

To see the second property, we consider some $2$-coloring $\psi$ of $G_0$ so that $G_0$ does not have a monochromatic copy of $H_{t,2}$. We show that this forces $R_0$ to be monochromatic. For a subset $S$ of the vertices and some vertex $v \not \in S$, define the \textit{color pattern $c_v$ with respect to $S$} to be the function with domain $S$ that maps a vertex $w\in S$ to the color of the edge $(v,w)$. This method was utilized in \cite{fglps}.

For a vertex $v\in F_1$, consider its color pattern $c_v$ with respect to $V(R_0)$. There are $2^n$ possible color patterns, so at least a $2^{-n}$ fraction of the vertices in $F_1$ have the same color pattern with respect to $V(R_0)$. Call the set of these vertices $S_1$. Then $|S_1| \ge 2^{-n}\cdot |V(F_1)|\ge \epsilon \cdot |V(F_1)|$, so there must exist a monochromatic copy $H_1$ isomorphic to $H_{t-1,1}$ in $S_1$. Without loss of generality, suppose $H_1$ is monochromatic in red. We claim that all the edges going from $S_1$ to $R_0$ (and in particular from $H_1$ to $R_0$) are blue. Indeed, since all vertices $v\in S_1$ have the same color pattern with respect to $R_0$, then for a fixed vertex $i\in R_0$ the edges $(i,v)$ have the same color for all $v\in S_1$. If that color is red, then $i$ along with all the vertices of $H_1$ would form a monochromatic red copy of $H_{t,2}$, which contradicts our definition of $\psi$. We now proceed inductively. Suppose we have identified red copies of $H_{t-1,1}$ labeled $H_1, \dots , H_{k-1}$ in $F_1, \dots ,F_{k-1}$ with vertex sets $V_1, \dots ,V_{k-1}$ respectively, and that all edges between these copies as well as to $R_0$ are blue. In $F_k$, at least a $2^{-n - t(k-1)}>\epsilon$ fraction of the vertices $S_k$ have the same color pattern with respect to $V(R_0)\cup V(H_1)\cup \dots \cup V(H_{k-1})$. Since $|S_k|>\epsilon \cdot |V(F_k)|$, we have $F[S_k]\to H_{t-1,1}$. Find a monochromatic copy of $H_{t-1,1}$ and call it $H_k$. Suppose $H_k$ is blue. Then, as in the case before, all the edges between $H_k$ and $R_0$, as well as to $H_1, \dots , H_{k-1}$, would have to be red, otherwise there would be a monochromatic blue copy of $H_{t,2}$. But if all these edges are red, then any vertex of $H_k$ along with $H_1$ forms a monochromatic copy of $H_{t,2}$, a contradiction. Thus, $H_k$ must be red, and consequently all edges between $H_k$ and $H_1, \dots , H_{k-1}, R_0$ must be blue, completing the inductive step. After applying this argument $t-2$ times, we have a collection $(H_1,\dots ,H_{t-2})$ of red copies of $H_{t-1,1}$ with complete bipartite blue graphs between any two of them. Now, suppose some edge in $R_0$ was blue. Then this edge, along with one vertex in each of $H_1, \dots , H_{t-2}$ and one other arbitrary vertex in $H_1$ forms a monochromatic blue copy of $H_{t,2}$. Thus, all the edges in $R_0$ must be colored red, as required.
\end{proof}

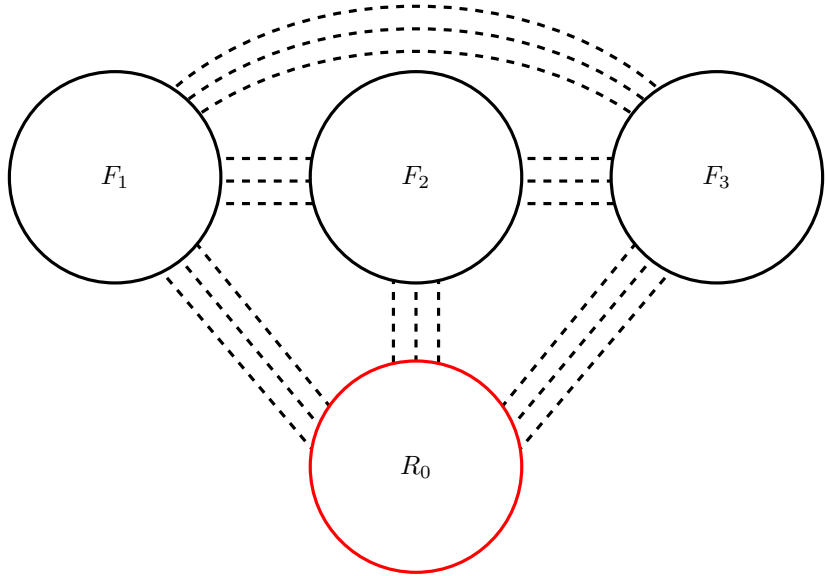
\begin{figure}[hd]
\begin{center}
\begin{tikzpicture}

\draw[dashed,very thick] (-4,0.5) -- (4,0.5);
\draw[dashed,very thick] (-4,0.8) -- (4,0.8);
\draw[dashed,very thick] (-4,1.1) -- (4,1.1);

\draw[dashed,very thick] (-4,0.5) to[out=60,in=120] (4,0.5);
\draw[dashed,very thick] (-4,0.8) to[out=60,in=120] (4,0.8);
\draw[dashed,very thick] (-4,1.1) to[out=60,in=120] (4,1.1);

\draw[dashed,very thick] (-3.6,0.8) -- (0,-3.6);
\draw[dashed,very thick] (-4,0.8) -- (0,-4);
\draw[dashed,very thick] (-4.4,0.8) -- (0,-4.4);

\draw[dashed,very thick] (3.6,0.8) -- (0,-3.6);
\draw[dashed,very thick] (4,0.8) -- (0,-4);
\draw[dashed,very thick] (4.4,0.8) -- (0,-4.4);

\draw[dashed,very thick] (-0.3,0.8) -- (-0.3,-4);
\draw[dashed,very thick] (0,0.8) -- (0,-4);
\draw[dashed,very thick] (0.3,0.8) -- (0.3,-4);

\tikzstyle{every node}=[very thick,draw,circle,fill=white,minimum size=80pt,
                            inner sep=0pt]
														
		\draw (-4,0.85) node (A) {$F_1$};
    \draw (0,0.85) node (B) {$F_2$};
    \draw (4,0.85) node (C) {$F_3$};
		
\tikzstyle{every node}=[very thick,draw=red,circle,fill=white,minimum size=80pt,
                            inner sep=0pt]
		\draw (0,-3) node (D) {$R_0$};

\end{tikzpicture}
\caption{Construction of the gadget graph $G_0$ for $t=5$ and $d=2$. The dashed lines represent complete bipartite graphs.}
\label{fig2}
\end{center}
\end{figure}

We now introduce a lemma which is a stronger version of an idea first introduced in \cite{bel} known as a positive signal sender.

\begin{lem}\label{easylemma}
There is a graph $G$ with two independent edges $e$ and $f$ so that, in any $2$-coloring of $G$ without a monochromatic copy of $H_{t,2}$, both edges $e$ and $f$ must have the same color. Furthermore, there is a $2$-coloring of $G$ with no red $H_{t,2}$ and no blue $K_t$ in which both edges $e$ and $f$ are red, and in which all of the edges incident to either of $e$ or $f$ are blue. Furthermore, there are no edges incident to both $e$ and $f$.
\end{lem}
\begin{proof}
This follows by taking $R_0$ in the previous lemma to be two disjoint edges, $e$ and $f$.
\end{proof}

We now take the above lemma and use it to prove a slight strengthening of itself.

\begin{lem}\label{DistanceTwo}
There is a graph $G$ with two independent edges $e$ and $f$ so that in any $2$-coloring of $G$ without a monochromatic copy of $H_{t,2}$ both edges $e$ and $f$ must have the same color. Furthermore, there is a $2$-coloring of $G$ with no red $H_{t,2}$ and no blue $K_t$ in which both edges $e$ and $f$ are red, and in which all of the edges incident to either of $e$ or $f$ are blue. Furthermore, any path between a vertex of $e$ and a vertex of $f$ has length at least $3$.\end{lem}
\begin{proof}
Lemma \ref{easylemma} gave us a graph that satisfied all of these constraints except for the last one. Take two copies $G',G''$ of this graph from Lemma \ref{easylemma}, with distinguished pairs of edges $(e',f')$ and $(e'',f'')$, respectively. Identify $f'$ with $e''$ and take $e=e'$ and $f=f''$, and call the resulting (combined) graph $G$. By construction, any path between a vertex of $e$ and a vertex of $f$ has length at least $3$. Also by construction, in any $2$-coloring of $G$ without a monochromatic copy of $H_{t,2}$, we must have that $e=e'$ and $f'$ have the same color, and $f'=e''$ and $f''=f$ have the same color, and so $e$ and $f$ have the same color. Finally, if we color $e,f',$ and $f$ all red, then we may extend this to colorings of $G'$ and $G''$ so that neither $G'$ nor $G''$ contains a red $H_{t,2}$ or a blue $K_t$ so that all edges incident to either of $e$ or $f$ are blue. This coloring contains no red $H_{t,2}$, as every connected component in red is contained entirely within at least one of $G'$ and $G''$, and neither one of these graphs has a red copy of $H_{t,2}$. There is no blue copy of $K_t$, as every blue triangle is contained either entirely within $G'$ or entirely within $G''$, and neither one contains a blue copy of $K_t$.
\end{proof}

The next lemma uses these so-called \textit{strong positive signal senders} to construct a weaker version of BEL gadgets for $H_{t,2}$. It is weaker because it does not guarantee that we can agree with a given coloring $\psi$ of a graph up to permutation of colors; it only guarantees that in a monochromatic $H_{t, 2}$-free coloring of the graph, the edges that are red in $\psi$ all end up with one color $\alpha_1$ and the edges that are blue in $\psi$ all end up with one color $\alpha_2$. The two colors $\alpha_1$ and $\alpha_2$ may be the same. After proving this lemma, we will then show that the existence of this weaker version of BEL gadgets implies the full strength of the BEL theorem, completing the proof.

\begin{lem} \label{BELHt2}
Given edge-disjoint graphs $G_0$ and $G_1$ on the same vertex set that are both $H_{t,2}$-free, there is a graph $G$ with an induced copy of $G_0 \cup G_1$ so that there is a $2$-coloring of $G$ without a monochromatic copy of $H_{t,2}$ in which $G_0$ is red and $G_1$ is blue. Furthermore, in any $2$-coloring of $G$ without a monochromatic $H_{t,2}$, all the edges in $G_0$ have the same color and all the edges in $G_1$ have the same color.
\end{lem}
\begin{proof}
Take $F$ to be a copy of the graph given by Lemma \ref{DistanceTwo}.

Form a graph $G$ as follows. Start with $G_0 \cup G_1$ on the same vertex set. Add two edges $e_0$ and $e_1$ independent from both $G_0$ and $G_1$. For any edge $f_0$ in $G_0$, we add a copy of $F$ with $e_0$ and $f_0$ as the distinguished edges. For any edge $f_1$ in $G_1$, we add a copy of $F$ with $e_1$ and $f_1$ as the distinguished edges. By construction, in any $2$-coloring of $G$ without a monochromatic $H_{t,2}$, all of the edges in $G_0$ have the same color and all of the edges in $G_1$ have the same color.

Consider coloring all edges of $G_0$ as well as $e_0$ red and all edges of $G_1$ as well as $e_1$ blue. By construction of $F$, we may extend this coloring to a coloring of $G$ in which every copy of $F$ attached to two edges in $G_0$ contains no blue $K_t$ and no red $H_{t,2}$ and in which all of the edges of $F$ that are incident to the two edges are blue. Symmetrically, in this coloring every copy of $F$ attached to two edges in $G_1$ contains no red $K_t$ and no blue $H_{t,2}$ and satisfies that all of the edges of $F$ that are incident to the two edges are red.

We claim there is no blue $H_{t,2}$. By symmetry it will follow that there is also no red $H_{t,2}$. First, observe that if we pick any two edges $(e,f)$ to which a copy of $F$ is attached, the vertices of any triangle in $G$ are either contained entirely in $F$ or entirely in the graph $G'$ obtained by removing the vertices of $F$ except $e$ and $f$; this follows immediately from the construction. Note further that any triangle that is not contained entirely in $G'$ must use some vertex $w$ that belongs to $F$ but not to $G'$; since there is no vertex in $F$ that has as a neighbor both a vertex of $e$ and a vertex of $f$, such a triangle may not use both a vertex of $e$ and a vertex of $f$; in particular, this means that all of the edges used by the triangle are contained in $F$ (note that there are no edges between $e$ and $f$ that are not contained in $F$, by the way we constructed $G_0$ and $G_1$). Therefore, any copy of $K_t$ must be contained entirely in the edges of $F$ or in entirely in $G'$. Since there is no blue $K_t$ in the copies of $F$ attached to edges from $G_0$, any blue copy of $H_{t,2}$ must have its copy of $K_t$ contained entirely in $G_1$ or entirely in some copy of $F$ attached to an edge of $G_1$. If we take a blue $K_t$ contained in some copy of $F$ attached an edge $e$ and some edge $f$ in $G_1$, then, since all of the edges incident to both $e$ and $f$ are red, if we take the connected component corresponding to the blue subgraph of $G$ containing this copy of $K_t$, we see that it is contained entirely in this copy of $F$. But by assumption this copy of $F$ has no $H_{t,2}$, and so this blue $K_t$ is not contained in any copy of $H_{t,2}$. Therefore, any blue copy of $H_{t,2}$ must have its $K_t$ contained in $G_1$. By assumption, $G_1$ contains no copy of $H_{t,2}$, so this copy must have some vertex outside of $G_1$ that has blue degree at least $2$ to this copy of $K_t$. Such a vertex cannot be contained in the copies of $F$ attached to an edge of $G_1$, as these are completely red to $G_1$. Therefore, this vertex must be contained in some copy of $F$ attached to an edge $e$ and an edge $f$ of $G_0$. But neither $e$ nor $f$ may be edges of the blue clique, since they are both red, and so this vertex must have a blue neighbor in $e$ and a blue neighbor in $f$, but this contradicts our assumptions on $F$, concluding the proof.
\end{proof}

If a graph $H$ satisfies the conclusions of the above lemma, we say it has \textit{weak BEL gadgets}. We now prove that this is enough to get \textit{strong BEL gadgets} for $H_{t,2}$, thus completing the proof of the upper bound.

\begin{lem}
If $H$ is connected and has weak BEL gadgets, then $H$ has BEL gadgets.
\end{lem}
\begin{proof}
Consider a graph $G$ with a given $2$-coloring $\psi$. Let $G$ be composed of the graphs $G'_0$ and $G'_1$, where $G_0'$ is the graph induced by the blue edges of $G$ and $G_1'$ is the graph induced by the red edges of $G$. Take $t$ to be the number of vertices in $H$. 

Define a graph $G_0$ by taking $G_0'$, adding to it some set $S$ of $t$ vertices, and adding edges to $S$ so it forms a copy of $H$ with one edge removed. Define $G_1$ by taking $G_1'$, adding to it $S$, and adding to $S$ the edge that was removed from $H$. We will show that this resulting graph can be made a strong BEL gadget for $H$. Note that neither $G_0$ nor $G_1$ contains a copy of $H$; the connected components are either connected components of $G_0$ or $G_1$, or are in $S$. Note further that in any $2$-coloring of $G_0 \cup G_1$ in which all of the edges in $G_0$ have the same color and all of the edges of $G_1$ have the same color, if $G_0$ and $G_1$ have the same color then there is a monochromatic copy of $H$, namely on vertex set $S$. Now, taking a weak BEL gadget for $G_0$ and $G_1$ yields the desired strong BEL gadget for $G_0'$ and $G_1'$.
\end{proof}

\section{Ramsey simple graphs}\label{TriangleFree}

The lower bound $s(H)\ge 2\delta (H)-1$ is established in \cite{fl}. A natural question that arises is to classify all graphs with $s(H)$ exactly equal to $2\delta(H)-1$.
\begin{defn}
A graph $H$ that satisfies $s(H)=2\delta(H)-1$ is called \textit{Ramsey simple}.
\end{defn}
In this section, we show a specific class of graphs to be Ramsey simple. We expand on the results of \cite{fglps2}, \cite{fl}, and \cite{szz}. In particular, we prove the following theorem.
\begin{thm}\label{CycleFree}
Let $H$ be a $d$-regular graph ($d \geq 1$) with BEL gadgets. Suppose there exists a vertex $v\in H$ for which $N(v)$ is an independent set and $H-v-N(v)$ is connected. Then $s(H)=2d - 1$.
\end{thm}
It is worth remarking that the technical constraints in the assumptions of the theorem, about having BEL gadgets and there being a vertex $v$ for which $H - v - N(v)$ is connected, are not very restrictive. In fact, recall that all $3$-connected graphs $H$ have BEL gadgets and note that a $d$-regular triangle-free graph chosen uniformly at random satisfies these constraints with high probability for fixed $d \geq 3$ and large enough $n$. That is, the theorem is applicable to almost all $d$-regular triangle-free graphs. 

For the rest of the section, let $H$ be a $d$-regular graph with BEL gadgets, where $d \geq 1$ and let $v$ be a vertex of $H$ with $N(v)$ an independent set and $H-v-N(v)$ connected. Our proof will be divided as follows. First, we will show that there exists an $H$-free graph $G$ with an independent set $S$ of size $2d -1$ so that adding an external vertex and connecting it to any $d$ vertices of $S$ creates a copy of $H$. Once we have constructed $G$, we will create a BEL gadget and conclude that $s(H) \le 2d - 1$, from which it follows by Theorem \ref{TrivialBound} that $s(H)=2d - 1$. Our proof roughly follows the ideas of \cite{fglps2}.
\begin{lem} \label{CycleFreeIndependentLemma}
There exists an $H$-free graph $G$ with an independent set $S$ of size $2d -1$ so that adding a vertex to $G$ and connecting it to any $d$ vertices of $S$ creates a copy of $H$. 
\end{lem}
\begin{proof}
Construct the graph $G$ as follows. Take an independent set $S$ of size $2d -1$. For any subset $S' \subseteq S$ of size $d$, construct a copy of the graph $H-v$, and then identify $N(v)$ and $S'$ (see Figure \ref{fig4}). Do this for all size-$d$ subsets $S' \subseteq S$, and call the resulting graph $G$. Formally, $G$ has vertex set $S\cup \left({\binom{S}{d}}\times [n-d-1]\right)$. Enumerate the vertices of $H$ as $v_1, \dots , v_n$ so that $v=v_n$ and $N(v)=\{v_{n-d},\dots , v_{n-1}\}$. For every set $S' \in \binom{S}{d}$, fix an arbitrary ordering of the vertices of $S'$ labeled $v^{S'}_{n-d},v^{S'}_{n-d+1},\dots , v^{S'}_{n-1}$. The edges of $G$ that are not incident to $S$ are pairs of the form $\{(S', k_1), (S',k_2)\}$ where $(v_{k_1}, v_{k_2})$ is an edge in $H - v$. The edges of $G$ that are incident to $S$ are pairs of the form $\{v^{S'}_{k_1}, (S',k_2)\}$ where $(v_{k_1}, v_{k_2})$ is an edge in $H - v$.

We claim that the graph $G$ is $H$-free. If it is not, some vertex that is not in $S$, i.e. some vertex of the form $(S',k)$, must be used in the copy of $H$ in $G$. Let $G'$ be the induced copy of $H-v$ corresponding to $S'$; i.e. $G' = G\left[S' \cup \left\{\{S'\} \times [n-d-1]\right\}\right]$. We claim that all vertices and edges of $G'$ must be contained in the copy of $H$. To see this, note that if any vertex $(S',k)$ is used in the copy of $H$, then all of its neighbors must be as well, for it only has $d$ neighbors and $H$ is $d$-regular. By connectivity of $H-v-N(v)$, this implies that all of the vertices of the form $(S',k)$ must be used. This in turn implies that all of the edges incident to any vertex of the form $(S',k)$ must be used, but this includes all edges and vertices of $G'$ since $G'$ has no isolated vertices and $S'$ is an independent set.

Since $G'$ only consists of $n-1$ vertices, there must be exactly one vertex $v'\not \in V(G')$ that is part of the copy of $H$. But any vertex not in $G'$ can have at most $d-1$ neighbors in $G'$. This is a contradiction, since $v'$ must have degree $d$, as $H$ is $d$-regular. Thus $G$ can contain no copy of $H$.
\end{proof}

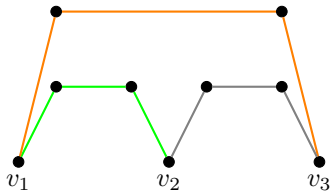
\begin{figure}[h!]
\begin{center}
\begin{tikzpicture}

\tikzstyle{every node}=[draw,circle,fill=black,minimum size=4pt,
                            inner sep=0pt]

    \draw (-2,0) node (x1) [label=below:$v_1$]{};
    \draw (0,0) node (x2) [label=below:$v_2$]{};
    \draw (2,0) node (x3) [label=below:$v_3$]{};
		
		\draw (-1.5,1) node (y12) {};
    \draw (-0.5,1) node (z12) {};
    \draw (0.5,1) node (y23) {};
    \draw (1.5,1) node (z23) {};
		\draw (-1.5,2) node (y13) {};
    \draw (1.5,2) node (z13) {};

\tikzstyle{every node}=[draw=none,circle,fill=none,minimum size=0pt,
                            inner sep=0pt]

\draw[green,thick] (x1) -- (y12);
\draw[green,thick] (y12) -- (z12);
\draw[green,thick] (z12) -- (x2);

\draw[orange,thick] (x1) -- (y13);
\draw[orange,thick] (y13) -- (z13);
\draw[orange,thick] (z13) -- (x3);

\draw[gray,thick] (x2) -- (y23);
\draw[gray,thick] (y23) -- (z23);
\draw[gray,thick] (z23) -- (x3);

\end{tikzpicture}
\caption{Construction of $G$ for the cycle $C_5$ (so that $d=2$). The set $\{v_1, v_2, v_3\}$ is an independent set of size $2d - 1$. For any vertex $v$ of a $5$-cycle, $C_5 - v$ is simply the path of length $3$. Each of the colored graphs shown is a copy of that. Note that this coloring has nothing to do with the $2$-coloring in red and blue that is the subject of this paper.}
\label{fig4}
\end{center}
\end{figure}

We finish the proof of Theorem \ref{CycleFree} now by constructing the graph $F$. We require that $F\to H$ with a vertex $v$ of degree $2d - 1$; furthermore, we require $F - v \not\to H$, which completes the proof.
\begin{proof}[Proof of Theorem \ref{CycleFree}]
There exist BEL gadgets for $H$ by assumption. Take two copies of the graph $G$ obtained from Lemma \ref{CycleFreeIndependentLemma}, and identify the two independent sets $S$ of size $2d - 1$. Color one copy of $G$ red and the other copy blue. Call this colored graph $G'$ with coloring $\psi'$. Construct a BEL gadget $F'$ so that $F'$ has an induced copy of $G'$ and satisfies the following property: $F' \not\to H$ and, in any coloring of $F'$ without a monochromatic copy of $H$, the induced copy of $G'$ has the coloring $\psi'$, up to permutation of colors. Add one more vertex $v$ to $F'$ and add edges from $v$ to all of $S$; call the resulting graph $F$. The degree of $v$ is $2d - 1$, and $F-v$ is not Ramsey for $H$. It only remains to prove that $F \to H$. Consider any $2$-coloring $\psi$ of the edges of $F$. If, in this coloring, there is a monochromatic copy of $H$ in $F - v$, we are done. Otherwise, we know that the induced copy of $G'$ has coloring $\psi'$. Observe that $v$ has degree $2d - 1$, and so by the pigeonhole principle, at least $d$ of the edges incident to $v$ must have the same color, say red. Then $v$, together with these $d$ neighbors in $S$ as well as the red copy of $H - v$ corresponding to these $d$ vertices, defines a monochromatic (red) copy of $H$. Therefore, $F \to H$ giving that $s(H) \le 2d - 1$. Together with the lower bound of Theorem \ref{TrivialBound}, this implies that $s(H)= 2d - 1$.
\end{proof}

\section{Conclusion and open problems}\label{openproblems}
We calculated the value of $s(H)$ for several classes of graphs, expanding on previous results. However, there remain several interesting related problems. 

Recall that a Ramsey simple graph is a graph $H$ for which $s(H) = 2\delta(H)-1$; in such a graph, $s(H)$  is described in terms of a simple graph parameter. We are particularly interested in the following question. Note that $G(n,p)$ is the graph obtained from $K_n$ by keeping every edge independently with probability $p$, and discarding it with probability $1 - p$.

\begin{qstn}
Fix any $0 < p < 1$. For sufficiently large $n$, can $s(G(n,p))$ be described with high probability in terms of some well-known or efficiently-computable graph parameter?
\end{qstn}

In Section \ref{TriangleFree}, we determined that the lower bound $s(H) \geq 2\delta(H)-1$ is exact when $H$ is a $3$-connected $d$-regular triangle-free graph subject to a minor technical constraint. The conjecture of \cite{fglps2} remains open.
\begin{conj} \label{TriangleConj}
For all connected triangle-free graphs $H$ on at least $2$ vertices, we have $s(H)=2\delta(H)-1$.
\end{conj}

A question of \cite{szz} also remains open.
\begin{qstn}\label{algorithm}
Given a graph $H$ as input, is there an efficient algorithm that computes $s(H)$?
\end{qstn}

It is easy to see that, for the class of graphs $H$ that have BEL gadgets, $s(H)$ is computable. This motivates the question of which graphs have BEL gadgets. The work of \cite{bnr} shows that all $3$-connected graphs have BEL gadgets. We showed here that $H_{t,2}$ has BEL gadgets. It is also known that cycles have BEL gadgets from \cite{g}. These observations motivate the following conjecture.

\begin{conj}
All $2$-connected graphs have BEL gadgets.
\end{conj}

Finally, we present a conjecture regarding the magnitude of $s(H)$. It bounds $s(H)$ in terms of the minimum degree of $H$ and the number of vertices; this conjecture is tight for both $K_k$ and $K_k \cdot K_2$.

\begin{conj}
If $H$ is a connected graph on $n$ vertices with minimum degree $\delta$, then $s(H) \leq \delta (n-1)$.
\end{conj}

\section{Acknowledgments}
We thank Prof. Jacob Fox, Dr. Tanya Khovanova, and Mr. Antoni Rangachev for useful comments and discussion.

\end{document}